\documentclass{amsart}

\usepackage{amssymb}
\usepackage{amsmath}
\usepackage{amscd}
\usepackage{ascmac}
\usepackage{graphicx}

\usepackage{macro}
\usepackage{here}

\usepackage{pinlabel}

\theoremstyle{plain}
\newtheorem{thm}{Theorem}[section]
\newtheorem{cor}[thm]{Corollary}
\newtheorem{lem}[thm]{Lemma}
\newtheorem{prop}[thm]{Proposition}

\theoremstyle{definition}
\newtheorem{defn}[thm]{Definition}

\theoremstyle{remark}
\newtheorem{rem}[thm]{Remark}

\begin{document}

\title{A note on signature of Lefschetz fibrations with planar fiber} 
\author[A.~Miyamura]{Akira Miyamura}
\address{Department of Mathematics, Tokyo Institute of Technology, 2-12-1 Oh-okayama, Meguro-ku, Tokyo 152-8551, Japan}
\email{miyamura.a.aa@m.titech.ac.jp}
\keywords{4-manifold, Lefschetz fibration, signature}
\date{August 1, 2017; MSC 2010: primary 57N13, secondary 55R05}

\maketitle

\begin{abstract}
Using theorems of Eliashberg and McDuff, Etnyre \cite{Et} proved that the intersection form of a symplectic filling of a contact 3-manifold supported by planar open book is negative definite.
 In this paper, we prove a signature formula for allowable Lefschetz fibrations over $D^2$ with planar fiber by computing Maslov index appearing in Wall's non-additivity formula.
 The signature formula leads to an alternative proof of Etnyre's theorem via works of Niederkr\"uger and Wendl \cite{NWe} and Wendl \cite{We}.
 Conversely, Etnyre's theorem, together with the existence theorem of Stein structures on Lefschetz fibrations over $D^2$ with bordered fiber by Loi and Piergallini \cite{LP}, implies the formula.
\end{abstract}

\section{Introduction}

The signature $\sigma(Y)$ of a smooth compact oriented $4$-manifold $Y$ is defined as the signature of the intersection form $Q_Y$ of $Y$.
 A signature formula for symplectic Lefschetz fibration over $S^2$ was studied by Smith \cite{S}.
 Endo \cite{En} gave a local signature formula for hyperelliptic Lefschetz fibrations over closed surfaces.
 Endo and Nagami \cite{EN} defined the signature of relators in the mapping class group of a fiber to obtain a signature formula of Lefschetz fibrations over $S^2$.
 An algorithm for computation of signature of Lefschetz fibrations over $D^2$ or $S^2$ was found by Ozbagci \cite{O}.
 Endo, Hasegawa, Kamada and Tanaka \cite{EHKT} showed that the signature of a Lefschetz fibration over a closed surface is equal to the signature of the corresponding chart. 
 There are many results about signature of Lefschetz fibrations with closed fiber.
 However, there are few such results for Lefschetz fibrations with bordered fiber.
 
In this paper, we prove a signature formula for allowable Lefschetz fibrations over $D^2$ with planar fiber by computing the Maslov index appearing in Wall's non-additivity formula \cite{Wa}.
 Let $\Sigma_{g,r}$ be a surface of genus $g$ with $r$ boundary components.
 Our main result is the following:
 


\begin{thm}\label{thm1}
Let $f:Y\rightarrow D^2$ be an allowable Lefschetz fibration with fiber $\Sigma_{0,r+1}$ and vanishing cycles $\gamma_1,\ldots,\gamma_m$ on $\Sigma_{0,r+1}$. 
Then the following equality holds:
\[
\sigma(Y)=-m+{\rm dim}\langle \gamma_1,\dots,\gamma_m\rangle.
\]
\end{thm}
Here $\langle \gamma_1,\dots,\gamma_m\rangle$ denotes the vector subspace of $H_1(\Sigma_{0,r+1};\R)$ generated by the homology classes of $\gamma_1,\dots,\gamma_m$.

\begin{rem}
Theorem \ref{thm1} holds only in the case that the fiber of a Lefschetz fibration is planar.
\end{rem}

Moreover, Theorem \ref{thm1} implies the following corollaries.
\begin{cor}\label{corr}
For $f:Y\rightarrow D^2$ as in Theorem \ref{thm1}, the signature $\sigma(Y)$ of $Y$ is equal to $-m+r-b_1(Y)$.
\end{cor}
\begin{proof}
The handle decomposition of $Y$ consists of one 0-handle, $r$\ 1-handles and $m$\ 2-handles.
 The 1-handles correspond to the generators of $H_1(\Sigma_{0,r+1};\R)$ and the boundary operator sends each 2-handle to the 1-chain of $Y$ corresponding to the homology class of a vanishing cycle. 
Hence $b_1(Y)$ coincides with $r-\textrm{dim}\langle \gamma_1,\dots,\gamma_m\rangle$.
 Therefore we have $\sigma(Y)=-m+r-b_1(Y)$ from Theorem \ref{thm1}.

\end{proof}

\begin{cor}\label{cor1}
For $f:Y\rightarrow D^2$ as in Theorem \ref{thm1}, we have $b_2^+(Y)=b_2^0(Y)=0$.
\end{cor}
\begin{proof}
By Corollary \ref{corr} and the definition of signature, we get
\[
\sigma(Y)=b_2^+(Y)-b_2^-(Y)=-m+r-b_1(Y).
\]
Since the Euler characteristic of $Y$ is related to the numbers of handles, we have
\[
1-b_1(Y)+b_2^+(Y)+b_2^-(Y)+b_2^0(Y)=1-r+m.
\]
Hence we get
\[
2b_2^+(Y)+b_2^0(Y)=0.
\]
Since $b_2^+(Y)$ and $b_2^0(Y)$ are non-negative, the both of them are zero.
\end{proof}

Finally, we show that Corollary \ref{cor1} implies the following theorem.

\begin{thm}{\rm (Etnyre \cite{Et})}\label{thmet}
If $X$ is a symplectic filling of a contact 3-manifold supported by a planar open book, then $b_2^+(X)=b_2^0(X)=0$.
\end{thm}
\begin{proof}
Niederkr\"uger and Wendl \cite{NWe} proved that any symplectic filling of a contact manifold supported by a planar open book is  symplectically deformation equivalent to a blow up of a Stein filling.
 Moreover, Wendl \cite{We} proved that any Stein filling of a contact manifold supported by such an open book admits an allowable Lefschetz fibration over $D^2$ with planar fiber.
 Since the blowing up preserves the negative definiteness, we get Theorem \ref{thmet} from Corollary \ref{cor1}.
\end{proof}

\begin{rem}
Theorem \ref{thmet} implies Theorem \ref{thm1}:
 By Loi and Piergallini \cite{LP}, a Lefschetz fibration $f:Y\rightarrow D^2$ as in Theorem \ref{thm1} admits a Stein structure.
 Therefore $Y$ is a symplectic filling of the boundary $\partial Y$.
 The restriction of an allowable Lefschetz fibration to the boundary is an open book of $\partial Y$, then $\partial Y$ is a contact 3-manifold supported by a planar open book.
 Applying Theorem \ref{thmet} to $Y$, we obtain Theorem \ref{thm1}.  
\end{rem}

This paper is organized as follows.
 In Section \ref{s2}, we review a definition of Lefschetz fibrations and the statement of Wall's non-additivity formula.
 We also give the construction of Lefschetz fibrations and its handle decomposition.
 The proof of Theorem \ref{thm1} is given in Section \ref{s3}.

\section{Preliminaries}\label{s2}
\subsection{Lefschetz fibration}
\ 

Let $Y$ be a smooth, compact and oriented 4-manifold.
\begin{defn}
A {\it Lefschetz fibration} is a smooth map $f:Y\rightarrow D^2$ such that:
\begin{itemize}
 \item[(1)]$\{b_1,\ldots,b_m\}\subset {\rm Int} D^2$ are the critical values of $f$, with $p_i$ a unique critical point of $f$ on $f^{-1}(b_i)$, for each $i$, and
 \item[(2)]about each $p_i$ and $b_i$, there are local complex coordinate charts centered at $p_i$\ and $b_i$ compatible with the orientations of
$Y$ and $D^2$ such that $f$ can be expressed as $f(z_1,z_2)=z_1^2+z_2^2$.
\end{itemize}
A Lefschetz fibration $f:Y\rightarrow D^2$\ is called {\it allowable} if each vanishing cycle of the Lefschetz fibration represents a 
non-trivial homology class of the fiber.
\end{defn}
 In this paper, we always assume that given Lefschetz fibrations are allowable.

For a given oriented surface $\Sigma$,
 we can construct a Lefschetz fibration over $D^2$ with fiber $\Sigma$ as follows.
 We start with a trivial $\Sigma$-bundle $pr_2:\Sigma\times D^2\rightarrow D^2$.
 We choose a point $b_0\in\textrm{Int}D^2$ and identify $pr_2^{-1}(b_0)$ with $\Sigma$.
 We choose distinct points $b_1,\dots,b_m\in\textrm{Int}D^2-\{b_0\}$ and properly embedded simple closed curves $\gamma_1,\dots,\gamma_m$ such that $\gamma_i$ is a curve on $pr_2^{-1}(b_i)$ for each $i$.
 Gluing $\Sigma\times D^2$ and a 4-dimensional 2-handle $h_1$ along $\gamma_1$ with framing $-1$ relative to the product framing of $h_1$, we obtain a 4-manifold $(\Sigma\times D^2)\cup h_1$.
 This manifold admits a Lefschetz fibration over $D^2$ with fiber $\Sigma$, vanishing cycle $\gamma_1$ and critical value $b_1$.
 Continuing this process for $\gamma_2,\dots,\gamma_m$ with all of the framings $-1$ relative to the product framings of 2-handles $h_2,\dots,h_m$ counter-clockwisely from $b_0$, the resulting manifold $(\Sigma\times D^2)\cup h_1\cup\dots\cup h_m$admits a Lefschetz fibration $f:(\Sigma\times D^2)\cup h_1\cup\dots\cup h_m\rightarrow D^2$ with fiber $\Sigma$, vanishing cycles $\gamma_1,\dots,\gamma_m$ and critical values $b_1,\dots,b_m$.
 We can naturally consider the smooth fiber bundle
\[
f\mid_{f^{-1}(D^2-\{b_1,\dots,b_m\})}:f^{-1}(D^2-\{b_1,\dots,b_m\})\rightarrow D^2-\{b_1,\dots,b_m\}\ .
\]
The fundamental group $\pi_1(D^2-\{b_1,\dots,b_m\},b_0)$ of $m$ punctured disk is generated by loops around each puncture and the \textit{local monodromy} is defined as the monodromy along each of such loops.
The local monodromy around $b_i$ is the right handed Dehn twist along $\gamma_i$ (for more detail, see the book \cite{GS} by Gompf and Stipsicz).

By the above construction, we obtain a handle decomposition of a Lefschetz fibration.
 The unique 0-handle,
 some 1-handles (depending on the genus and the number of boundary components of the fiber) and some 2-handles.
  It is useful to compute the Euler charecteristic of the Lefschetz fibration and we use this in Section \ref{s3}.

\subsection{Wall's non-additivity}\label{wallsnonadd}
\ 

Now we explain a formula called Wall's non-additivity formula.
\\Let $Y,Y_-,Y_+$\ be $4$-manifolds, $X_0,X_-,X_+$\ be $3$-manifolds and $Z$\ be a $2$-manifold such that
\\$Y=Y_-\cup Y_+,$
\\$Y_-\cap Y_+=X_0,$
\\$\partial Y_\pm=X_\pm\cup X_0,$
\\$\partial X_\pm=\partial X_0=Z.$(Figure \ref{wallna})
\begin{figure}[ht!]
\labellist
\hair 2pt
\pinlabel $X_0$ [t] at 120 55
\pinlabel $X_+$ [t] at 35 85
\pinlabel $X_-$ [t] at 35 7
\pinlabel $Y_+$ [t] at 167 63
\pinlabel $Y_-$ [t] at 167 26
\pinlabel $Z$ [t] at -10 45
\pinlabel $Z$ [t] at 250 45
\endlabellist
\centering
\includegraphics{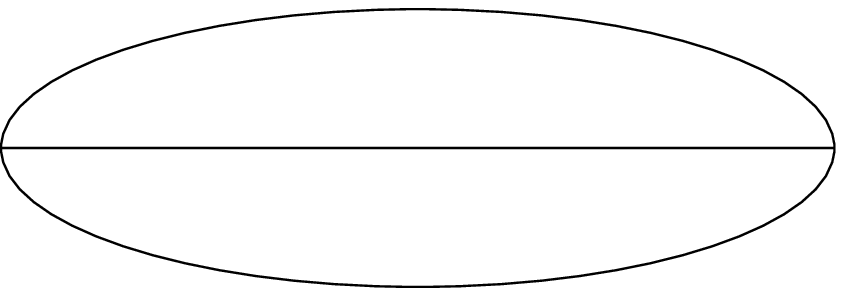}
\caption{}
\label{wallna}
\end{figure}
\\$Y$ induces the orientations of $Y_\pm$, and otherwise:
\\$\partial[Y_-]=[X_0]-[X_-],\ \partial[Y_+]=[X_+]-[X_0],\ \partial[X_0]=\partial[X_\pm]=[Z]$.

For each $\epsilon=0,+\ {\rm and}\ -$, let $i_\epsilon$\ be the inclusion $Z=\partial X_\epsilon\hookrightarrow X_\epsilon$ and $i_{\epsilon*}$ be the induced
map of $i_\epsilon$ to the first homology group. Then we define $V,L_\epsilon$ and $W$ as follows:
\begin{align}
&V=H_1(Z;\R),\nonumber\\
&L_\epsilon={\rm Ker}(i_\epsilon*),\nonumber\\
&W=\frac{L_-\cap(L_0+L_+)}{(L_-\cap L_0)+(L_-\cap L_+)}.\nonumber
\end{align}
Let $Q_Z$ be the intersection form of $Z$\ and we consider
 the bilinear map $\Psi':L_-\cap(L_0+L_+)\times L_-\cap(L_0+L_+)\rightarrow \R:\Psi'(a,a')=Q_Z(a,b')$. 
Here $a$ and $a'$ are elements of
 $L_-\cap(L_0+L_+)$
and $b'$\ is an element of $L_0$\ satisfying $a'+b'+c'=0$ for some $c'\in L_+$. 
One can see that $\Psi'$ induces a well-defined symmetric bilinear map $\Psi:W\times W\rightarrow \R$.
The signature of $\Psi$ is denoted by $\sigma(V;L_-,L_0,L_+)$. 
We also write the signature of a $4$-manifold $M$ as $\sigma(M)$.
 Wall \cite{Wa} proved the following theorem:
\begin{thm}{\rm (Wall \cite{Wa})}
\ $\sigma(Y)=\sigma(Y_-)+\sigma(Y_+)-\sigma(V;L_-,L_0,L_+)$
\end{thm}

We mainly use Wall's formula to prove Theorem \ref{thm1}.

\section{Proof of main theorem}\label{s3}
In this section, we consider homology group with coefficient $\R$ and denote the homolgy classes of the curves on surfaces by the same symbols of those curves. 

Let $f:Y_+\rightarrow D^2$ be a Lefschetz fibration with fiber $\Sigma_{0,r+1}$ and vanishing cycles $\gamma_1,\ldots,\gamma_m$ on $\Sigma_{0,r+1}$.  
 Since local monodromies of a Lefschetz fibration fix the boundaries of regular fiber point wise, the smooth fiber bundle  
\[
f\mid_{\partial Y_+-{\rm Int}\ f^{-1}(\partial D^2)}:\partial Y_+-{\rm Int}\ f^{-1}(\partial D^2)\rightarrow D^2
\]
is a trivial 
$\underset{i=0}{\overset{r}{\coprod}}\partial D^2_i$-bundle over $D^2$. Here $D^2_i$ is a copy of $D^2$ for each $i$. 
We take a bundle isomorphism $\Phi:(\underset{i=0}{\overset{r}{\coprod}}\partial D^2_i)\times D^2\rightarrow\partial Y_+-{\rm Int}\ f^{-1}(\partial D^2)$
and define $Y,Y_-,X_0,X_\pm$ and $Z$ as follows:
\\$Y_-=(\underset{i=0}{\overset{r}{\coprod}}D^2_i)\times D^2,$
\\$X_+=f^{-1}(\partial D^2),$
\\$X_0=(\underset{i=0}{\overset{r}{\coprod}}\partial D^2_i)\times D^2\approx\partial Y_+-{\rm Int}\ f^{-1}(\partial D^2)
=\partial Y_+-{\rm Int}\ X_+,$
\\$X_-=(\underset{i=0}{\overset{r}{\coprod}} D^2_i)\times \partial D^2,$
\\$Y=Y_+\cup_\Phi Y_-,$
\\$Z=\partial X_0=(\underset{i=0}{\overset{r}{\coprod}}\partial D^2_i)\times \partial D^2,$
\\where $Y_+\cup_\Phi Y_-$ is the manifold obtained by gluing of $Y_+$ and $Y_-$ by $\Phi$. 

Under the above assumption, the following two propositions hold.

\begin{prop}\label{p1}
{\rm dim} $W=\sigma(V;L_-,L_0,L_+)$.
\end{prop}
\begin{prop}\label{p2}
{\rm dim} $W=$ {\rm dim} $\langle\gamma_1,\dots,\gamma_m\rangle$
\end{prop}

We see that Theorem \ref{thm1} follows from Propositions \ref{p1} and \ref{p2}.
\begin{proof}[Proof of Theorem \ref{thm1}]
By the above construction, $f$ can be naturally extended  to a Lefschetz fibration over $Y$ with regular fiber $S^2$.
 More precisely, $Y$ is diffeomorphic to $D^2\times S^2\#m\overline{\C\mathbb{P}^2}$.
 This implies that the signature of $Y$ is $-m$.
 Since the signature of $Y_-$ is obviously $0$, we have
\[
-m=\sigma(Y_+)-\sigma(V;L_-,L_0,L_+)=\sigma(Y_+)-\textrm{dim}\langle\gamma_1,\dots,\gamma_m\rangle
\]
from Wall's non-additivity formula and Propositions \ref{p1} and \ref{p2}.

\end{proof}

 We choose and fix a point $p_i$ in $\partial D^2_i$ for each $i=0,\dots,r$.
 We define the homology classes $m_i$ and $l_i$ in $H_1(Z)$ as 
$[\partial D^2_i\times\{*\}]$ and $[\{p_i\}\times\partial D^2]$, respectively,
 where $*$ is the point correspond to $0\in\R\slash\Z\approx\partial D^2$.
 Under this notation, we always assume that $m_i$ and $l_j$ are oriented such that $Q_Z(m_i,l_j)$ is $\delta_{ij}$. 
 
Before the proof of the Proposition \ref{p1} and Proposition \ref{p2}, we introduce a lemma.

\begin{lem}\label{lem1}
The vector space $W$ is spaned by
\[
\bigl\{\ \underset{s=1}{\overset{m}{\sum}}Q_Z(\gamma_s,l_1)\gamma_s,\dots,\underset{s=1}{\overset{m}{\sum}}Q_Z(\gamma_s,l_r)\gamma_s\bigr\}.
\]
For this generating set, $\Psi$ satisfies
\[
\Psi\left(\underset{s=1}{\overset{m}{\sum}}Q_Z(\gamma_s,l_i)\gamma_s\ ,\underset{s=1}{\overset{m}{\sum}}Q_Z(\gamma_s,l_j)\gamma_s\right)
=\underset{s=1}{\overset{m}{\sum}}Q_Z(\gamma_s,l_i)Q_Z(\gamma_s,l_j).
\]
\end{lem}
We prove Proposition \ref{p1} by using Lemma \ref{lem1}.

\begin{proof}[Proof of Proposition \ref{p1}]
We note that Proposition \ref{p1} is equivalent to the positive definiteness of $\Psi$. 
Let $\alpha_j$ be the homology class $\underset{s=1}{\overset{m}{\sum}}Q_Z(\gamma_s,l_j)\gamma_s$ and dim $W=l$.
 Since $\{\alpha_1,\dots\alpha_r\}$ generates $W$, we can write the basis of $W$ as $\{\alpha_{i_1},\dots,\alpha_{i_l}\}$ for some $1\leq i_1<\dots<i_l\leq r$.
 We take an arbitrary $u\in W$ and write it as 
$\underset{j=1}{\overset{l}{\sum}}c_j\alpha_{i_j}\ (c_j\in\R)$. Then we get
\begin{align}
\Psi(u,u)&=\Psi(\ \underset{j=1}{\overset{l}{\sum}}c_j\alpha_{i_j},\underset{k=1}{\overset{l}{\sum}}c_k\alpha_{i_k})
=\underset{j,k}{\sum}c_jc_k\Psi(\alpha_{i_j},\alpha_{i_k})\nonumber\\
&=\underset{j,k}{\sum}c_jc_k\underset{s=1}{\overset{m}{\sum}}Q_Z(\gamma_s,l_{i_j})Q_Z(\gamma_s,l_{i_k})\nonumber\\
&=\underset{s=1}{\overset{m}{\sum}}\left(\underset{j=1}{\overset{l}{\sum}}c_jQ_Z(\gamma_s,l_{i_j})\right)^2\geq0\nonumber
\end{align}
Therefore, if we assume $\Psi(u,u)=0$, $\underset{j=1}{\overset{l}{\sum}}c_jQ_Z(\gamma_s,l_{i_j})$ is 0 for each $s\in\{1,\dots,m\}$.
 Then we get
\[
u=\underset{j=1}{\overset{l}{\sum}}c_j\alpha_{i_j}=
\underset{j=1}{\overset{l}{\sum}}c_j\underset{s=1}{\overset{m}{\sum}}Q_Z(\gamma_s,l_{i_j})\gamma_s
=\underset{s=1}{\overset{m}{\sum}}\left(\underset{j=1}{\overset{l}{\sum}}c_jQ_Z(\gamma_s,l_{i_j})\right)\gamma_s=0.
\]
This means that $\Psi$ is positive definite.

\end{proof}

Next, we prove Proposition \ref{p2}.
\begin{proof}[Proof of Proposition \ref{p2}]
We set $\{e_1,\dots,e_r\}$ the standard basis of $\R^r$ and define the isomorphism
 $\phi:\underset{i=1}{H_1(\Sigma_{0,r+1})=\overset{r}{\oplus}}\R m_i\rightarrow\R^r:m_i\mapsto e_i.$
Let $A$ and $B$ be the matrices defined by $(\phi(\alpha_1),\dots,\phi(\alpha_r))$ and $(\phi(\gamma_1),\dots,\phi(\gamma_m))$, respectively.
The statement that the dimension of $W$ is equal to the dimension of $\langle\gamma_1,\dots,\gamma_m\rangle$ is equiverent to that the rank of $A$ is equal to the rank of $B$. Now we write $\gamma_s$ by $\underset{i=1}{\overset{r}{\sum}}x_{is}m_i\ (x_{is}\in\R)$,
 we have
\[
Q_Z(\gamma_s,l_k)=Q_Z\left(\underset{i=1}{\overset{r}{\sum}}x_{is}m_i,l_k\right)=\underset{i=1}{\overset{r}{\sum}}x_{is}\delta_{ik}=x_{ks}.
\]
Recall that, by definition, the intersection number of $m_i$ and $l_k$ in $Z$ is $\delta_{ik}$. Then we get the expression for
 $\alpha_j$ as $\underset{i=1}{\overset{r}{\sum}}\left(\underset{s=1}{\overset{m}{\sum}}x_{js}x_{is}\right)m_i$ with respect to the generating set $\{m_1,\dots,m_r\}$ and this implies that 
\[
A=\left(\underset{s=1}{\overset{m}{\sum}}x_{js}x_{is}\right)_{i,j}.
\]
On the other hand, this is written as
\[
B\hspace{1.5pt}{}^t\!B=
\begin{pmatrix}
x_{11}&x_{12}&\ldots&x_{1m}\\
x_{21}&x_{22}&\ldots&x_{2m}\\
\vdots&\vdots&\ddots&\vdots\\
x_{r1}&x_{r2}&\ldots&x_{rm}\\
\end{pmatrix}
\begin{pmatrix}
x_{11}&x_{21}&\ldots&x_{r1}\\
x_{12}&x_{22}&\ldots&x_{r2}\\
\vdots&\vdots&\ddots&\vdots\\
x_{1m}&x_{2m}&\ldots&x_{mr}\\
\end{pmatrix}
=\left(\underset{s=1}{\overset{m}{\sum}}x_{is}x_{js}\right)_{i,j}=A.
\]
For any real matrix $M$, the rank of $M$ coincides with that of $M\transp M$ (See \cite{L} by Liebeck).
 Thus Proposition \ref{p2} holds.
\end{proof}

Next, we introduce the following lemma for the proof of Lemma \ref{lem1}. 

Let $\Sigma$ be the surface $\Sigma_{g,r+1}$ for $g\in\Z_{\geq0}$.
 The boundary of $\Sigma$ can be identified with the disjoint union of $r+1$ copies of $S^1$.
 The boundary of $D_i^2$ can be also identified with $S^1$.
 The point $p_i\in\partial D^2_i$, defined in the first of this section, can be considered as a point of $\partial\Sigma$. 
 We set proper embedded arcs $\sigma_{1},\dots,\sigma_{r}$ which connect $p_0$ to $p_1,\dots,p_r$, respectively. 
For a given orientation preserving diffeomorphism $\phi$ of $\Sigma$, $\Sigma_\phi$ denotes the quotient space $\Sigma\times[0,1]/(x,1)\sim(\phi(x),0)$.
 We take a set $\{a_i,b_j,m_k\}$ of generators of $H_1(\Sigma)$ such that $Q_{\Sigma}(a_i,b_j)=\delta_{ij}$ and $Q_{\Sigma}(m_k,\cdot)=0$ for $i,j=1,\dots,g$ and $k=0,\dots,r$\ (Fig \ref{fig2}). 
\begin{lem}\label{lem2}
Let $\gamma_1,\dots,\gamma_m$ be simple closed curves in {\rm Int} $\Sigma$ and $\phi$ the product of Dehn twists $D(\gamma_m)\circ\dots\circ D(\gamma_1)$. 
We express the first homology group of $\Sigma_\phi$ as $\langle a_1,b_1,\dots,a_g,b_g,m_1,\dots,m_r,l_0\mid a_i=\phi_*(a_i),b_i=\phi_*(b_i)\ (i=1,\dots,g)\rangle$ so that the inclusion map $i$ from $\partial\Sigma_{\phi}$ to $\Sigma_{\phi}$ satisfying the following conditions:
\begin{enumerate}
\item The set $\{m_0,l_0,\dots,m_r,l_r\}$ is a basis of $H_1(\partial\Sigma_\phi)$ and satisfies $Q_{\partial\Sigma_{\phi}}(m_i,l_j)=\delta_{ij}$, 
\item $i_*(m_0)=-\underset{j=1}{\overset{r}{\sum}}m_j$,
\item $i_*(m_j)=m_j\ (j\neq0)$,
\item $i_*(l_0)=l_0$.
\end{enumerate}
Then we have
\[
i_*(l_j)=l_0+[\phi_*(\sigma_j)-\sigma_j]\ (j\neq0).
\]
If g is 0, this can be written as:
\[
i_*(l_j)=l_0-\underset{s=1}{\overset{m}{\sum}}Q_{\partial\Sigma_\phi}(\gamma_s,l_j)\gamma_s\ (j\neq0),
\]
where $i_*$ denotes the induced map of the inclusion $i$ to the homology group.
\end{lem}
\begin{rem}
For the case $g=0$, the basis of $H_1(\partial\Sigma_\phi)$ and $H_1(\Sigma_{0,r+1})$ are $\{m_j,l_j\mid j=0,\dots,r\}$ and $\{m_j\mid j=1,\dots,r\}$, respectively.
 Thus we can naturally think of $H_1(\Sigma_{0,r+1})$ as a subgroup of $H_1(\partial\Sigma_\phi)$.
 Thus $Q_{\partial\Sigma_\phi}(\gamma_s,l_j)$ can be defined. 
\end{rem}

\begin{figure}[ht!]
\labellist
\hair 2pt
\pinlabel $a_1$ [t] at 65 145
\pinlabel $b_1$ [t] at 98 99
\pinlabel $a_g$ [t] at 215 162
\pinlabel $\vdots$ [l] at 300 62
\pinlabel $b_g$ [t] at 243 113
\pinlabel $\dots$ [r] at 143 72
\pinlabel $m_0$ [l] at 315 123
\pinlabel $m_1$ [l] at 315 92
\pinlabel $m_r$ [l] at 315 23
\endlabellist
\centering
\includegraphics{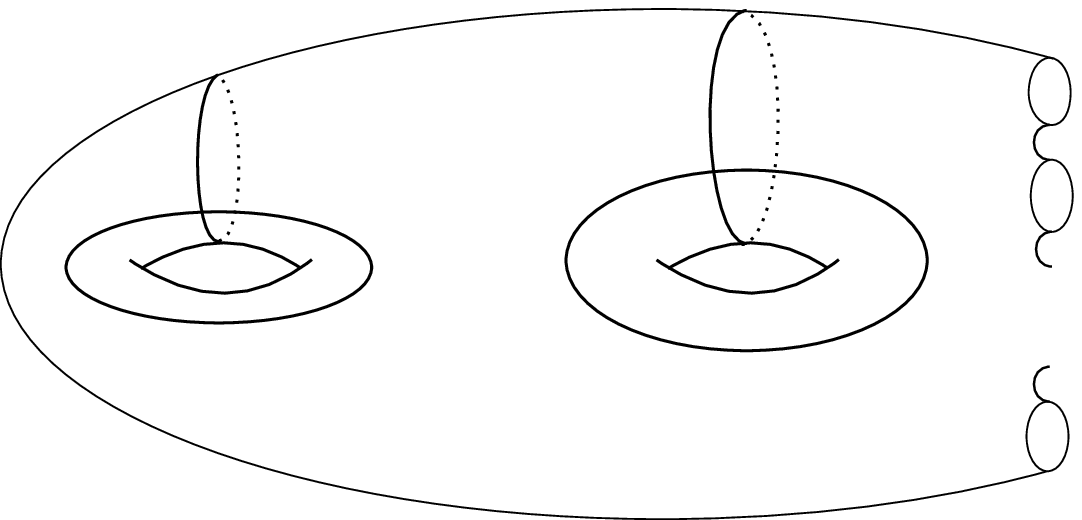}
\caption{}
\label{fig2}
\end{figure}

We prove Lemma \ref{lem1} by using Lemma \ref{lem2}.
\begin{proof}[Proof of Lemma \ref{lem1}]
We already note that $H_1(Z)$ is generated by $m_j$ and $l_j$ in the first part of this section, and by the definitions, $L_{0}$ and $L_{-}$
is given by:
\begin{align}
&L_-=\mbox{Ker}(i_{-*}:H_1(Z)\rightarrow H_1(X_-))
=\underset{i=0}{\overset{r}{\oplus}}\mathbb{R}m_i,\nonumber\\
&L_0=\mbox{Ker}(i_{0*}:H_1(Z)\rightarrow H_1(X_0))=\underset{i=0}{\overset{r}{\oplus}}\mathbb{R}l_i.\nonumber
\end{align}

By the assumption, the 3-manifold $X_+$ is $\Sigma_{\phi}$ in Lemma \ref{lem2} of the case $g=0$.
Then we apply Lemma \ref{lem2}, we get
\[
L_+={\rm Ker}\ i_{+*}=\langle\ l_1-l_0+\underset{s=1}{\overset{m}{\sum}}Q_Z(\gamma_s,l_1)\gamma_s,\ldots
,l_r-l_0+\underset{s=1}{\overset{m}{\sum}}Q_Z(\gamma_s,l_r)\gamma_s,\underset{i=0}{\overset{r}{\sum}}m_i\ \rangle.
\]
By these results, the subspaces $L_-\cap(L_0+L_+),L_-\cap L_0$ and $L_+$ of $H_1(Z)$ are given by:
\begin{align}
&L_-\cap(L_0+L_+)=\langle\ \underset{s=1}{\overset{m}{\sum}}Q_Z(\gamma_s,l_1)\gamma_s,\ldots,
\underset{s=1}{\overset{m}{\sum}}Q_Z(\gamma_s,l_r)\gamma_s,\underset{i=0}{\overset{r}{\sum}}m_i\ \rangle,\nonumber\\
&L_-\cap L_0=\{0\},\ L_-\cap L_+=\mathbb{R}\left(\underset{i=0}{\overset{r}{\sum}}m_i\right).\nonumber
\end{align}
Therefore the vector space $W$, the quotient of $L_-\cap(L_0+L_+)$ by $L_-\cap L_0+L_-\cap L_+$, can be writen as that in the statement.

Next we consider the bilinear form $\Psi$. 
For each $j$, we denote the element $\underset{s=1}{\overset{m}{\sum}}Q_Z(\gamma_s,l_j)\gamma_s$ by $\alpha_j$. 
By the obvious facts that  
$\alpha_j\in L_-,\ l_j-l_0\in L_0,\ l_j-l_0+\alpha_j\in L_+$ and $\alpha_j+(l_j-l_0)-(l_j-l_0+\alpha_j)=0$, $\Psi$ satisfies 
\[
\Psi(\alpha_i,\alpha_j)=Q_Z(\alpha_i,l_j-l_0)=Q_Z(\alpha_i,l_j)=\underset{s=1}{\overset{m}{\sum}}Q_Z(\gamma_s,l_i)Q_Z(\gamma_s,l_j).
\]

\end{proof}

We lastly prove Lemma \ref{lem2}.
\begin{proof}[Proof of Lemma \ref{lem2}]
The expression of $H_1(\Sigma_\phi)$ can be easily induced by the long exact sequence proved in Hatcher's book \cite{H} at Example 2.48:

Let $L$ and $M$ be topological spaces and $f$ and $g$ maps from $L$ to $M$. 
Define an equivalence relation $\sim$ on $L\times[0,1]\cup M$ by $(x,0)\sim f(x),(x,1)\sim g(x)$ and $N$ by the quotient of $L\times[0,1]\cup M$ by $\sim$. Then the following sequence is exact:
\[
\cdots\rightarrow H_n(L)\xrightarrow{f_*-g_*}H_n(M)\xrightarrow{i_*}H_n(N)\rightarrow H_{n-1}(L)\rightarrow\cdots\label{hatcher},
\]
where $i_*$ denotes the induced map of $M\hookrightarrow N$.

We use this sequence for $L=M=\Sigma,f=\textrm{id}$ and $g=\phi$.
 Consequently, $N$ becomes the mapping torus $\Sigma_\phi$. 

Next we prove the statement for fixed $j\in\{1,\dots,r\}$. 
We denote by $\pi$ the projection $\Sigma\times[0,1]\rightarrow\Sigma_\phi$.
 Let $c$ be a simplicial 2-chain expressing $\sigma_j\times[0,1]\subset\Sigma\times[0,1]$ and $l_0'$ and $l_j'$ simplicial 1-chains expressing $\{p_0\}\times[0,1]$ and $\{p_j\}\times[0,1]\subset\Sigma\times[0,1]$, respectively. We orient these chains so that 
$\partial c=\sigma_j\times\{1\}-l_j'-\sigma_j\times\{0\}+l_0'$ (Fig \ref{fig3}).
\begin{figure}[ht!]
\labellist
\hair 2pt
\pinlabel $\sigma_j\times\{0\}$ [t] at 85 40
\pinlabel $\vdots$ [t] at 35 131
\pinlabel $\vdots$ [t] at 35 83
\pinlabel $\vdots$ [t] at 310 131
\pinlabel $\vdots$ [t] at 310 83
\pinlabel $l_j'$ [t] at 167 113
\pinlabel $l_0'$ [t] at 167 26
\pinlabel $\Sigma\times\{1\}$ [l] at 280 -9
\pinlabel $\Sigma\times\{0\}$ [l] at 5 -9
\pinlabel $\sigma_j\times\{1\}$ [t] at 360 40
\endlabellist
\centering
\includegraphics[scale=0.8]{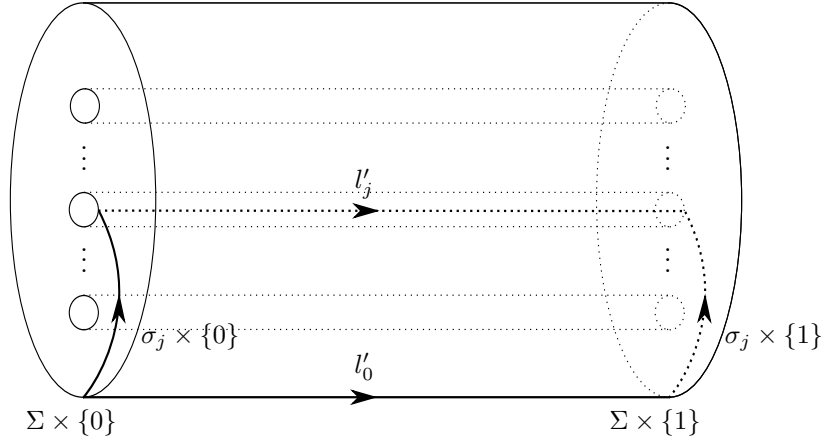}
\caption{$\Sigma\times[0,1]$ and simplicial chains}
\label{fig3}
\end{figure}
Then the homology class $[\sigma_j\times\{1\}-l_j'-\sigma_j\times\{0\}+l_0']$ is zero in $H_1(\Sigma\times[0,1])$. 
By the definition of $\Sigma_\phi$, the class represented by $\sigma_j\times\{1\}$ is equal to the class represented by 
$\phi_\#(\sigma)\times\{0\}$ in $H_1(\Sigma_\phi)$. Here $\phi_\#$ is the induced map of $\phi$ to the simplicial chain complex. 
Hence we get 
\[
i_*(l_j)=[\pi_\#(l_j')]=[\phi_\#(\sigma_j)\times\{1\}-\pi_\#(\sigma_j\times\{0\})+\pi_\#(l_0')]=i_*(l_0)+[\phi_\#(\sigma_j)-\sigma_j]
\]

For the case that $g$ is 0, we will prove
\[
[\phi_\#(\sigma_j)-\sigma_j]=-\underset{s=1}{\overset{m}{\sum}}Q_Z(\gamma_s,l_j)\gamma_s.
\]
We take connected neighborhoods $\nu_k$ of $p_k$ in $\partial\Sigma_{0,r+1}$ for $k=0,j$. Attaching a $2$-dimensional $1$-handle 
$h_1=[-1,1]\times[0,1]$ along $\{-1\}\times[0,1]$ and $\{1\}\times[0,1]$ together with $\nu_0$ and $\nu_j$, we obtain a new surface $\Sigma'$ diffeomorphic to $\Sigma_{1,r}$(Fig \ref{fig4}).

\begin{figure}[H]
\labellist
\hair 2pt
\pinlabel $\sigma_j$ [t] at 90 84
\pinlabel $\sigma_j'$ [t] at 83 209
\pinlabel $p_0$ [t] at 35 148
\pinlabel $p_j$ [t] at 172 148
\pinlabel $\Sigma_{0,r+1}$ [t] at 216 50
\pinlabel $\ldots$ [t] at 267 170
\pinlabel $h_1$ [t] at 28 238
\endlabellist
\centering
\includegraphics[scale=0.7,height=7.7cm]{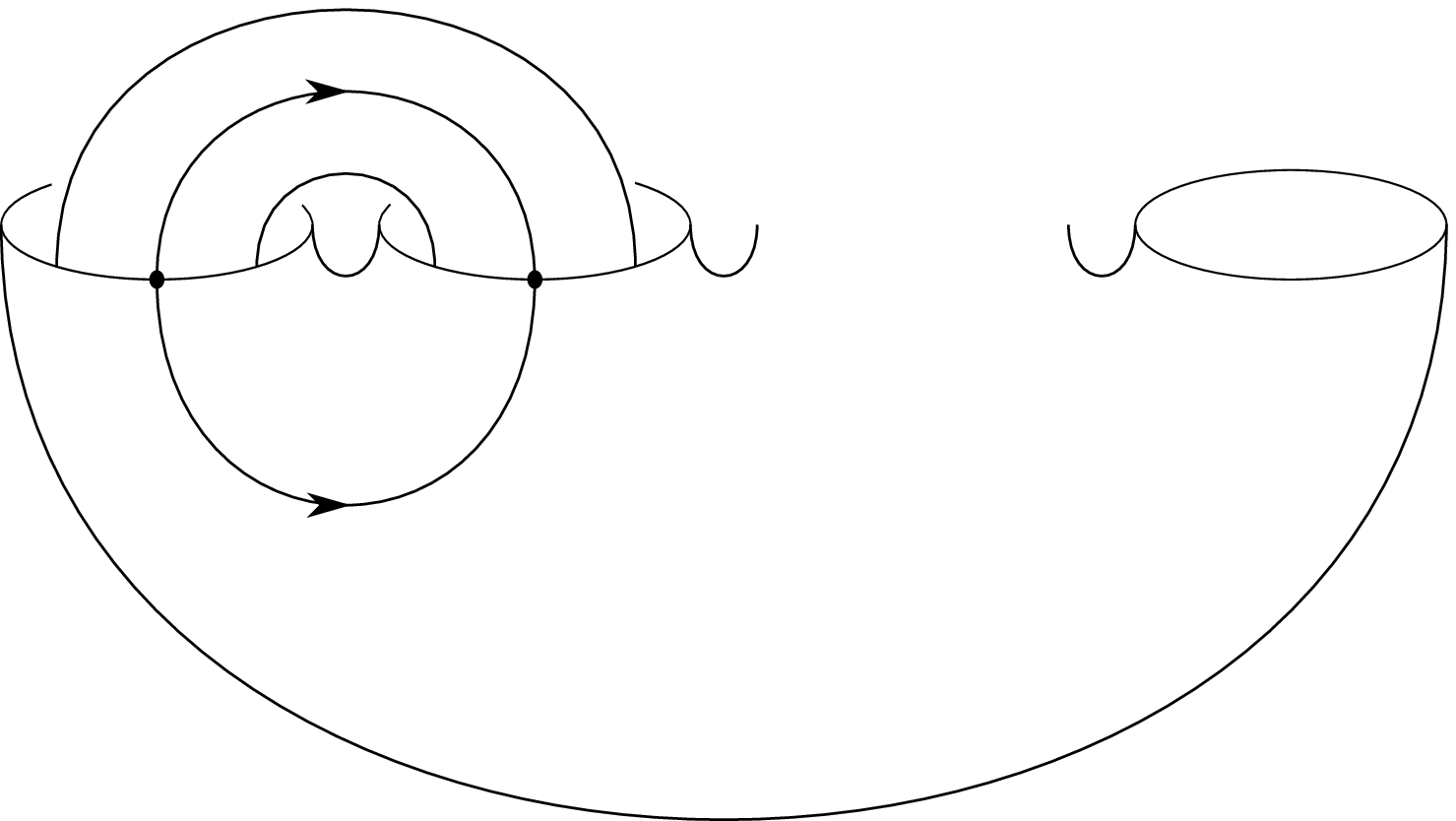}
\caption{$\Sigma'$}
\label{fig4}
\end{figure}
Here $\sigma_j'$ denotes a simplicial 1-chain expressing the core of the $h_1$ in $\Sigma'$.
 We note that $H_1(\Sigma')$ is spaned by $\{[\sigma_j-\sigma_j'],m_1,\dots,m_r\}$ and the intersection form $Q_{\Sigma'}$ satisfies 
\[
Q_{\Sigma'}(x,y)=
\left\{
\begin{array}{l}
1\hspace{0.7cm} \textrm{if}\ (x,y)=([\sigma_j-\sigma_j'],m_j)\\
0\hspace{0.7cm} \textrm{if}\ \{x,y\}\neq\{m_j,[\sigma_j-\sigma_j']\}.
\end{array}
\right.
\]
For a compact oriented suface $F$, every orientation preserving diffeomorphism of $F$ is isotopic to a product of Dehn twists. 
Let $D(\gamma)$ be the right-handed Dehn twist along a simple closed curve $\gamma$. 
The action of $D(\gamma)$ on $H_1(F)$ is given by the \textit{Picard-Lefschetz formula}:
$D(\gamma)_*(x)=x+Q_F(\gamma,x)\gamma$ for $x\in H_1(F)$.

We will first prove the case that $\phi=D(\gamma)$, and then the case of that $\phi$ is a product of Dehn twists. 
Let $\gamma$ be a simple closed curve in $\Sigma_{0,r+1}\subset\Sigma'$. 
By the fact that $\sigma'$ is not in Int $\Sigma_{0,r+1}$ and the Picard-Lefschetz formula, we have
\[
[D(\gamma)_{\#}(\sigma_j)-\sigma'_j]=D(\gamma)_*([\sigma_j-\sigma'_j])=[\sigma_j-\sigma'_j]+Q_{\Sigma'}(\gamma,[\sigma_j-\sigma'_j])\gamma.
\]
Then $\gamma$ is a homology class in $H_1(\Sigma_{0,r+1})=\underset{j=1}{\overset{r}{\oplus}}\R m_j$, it can be expressed by
$\underset{j=1}{\overset{r}{\sum}}a_jm_j$ for some $a_j\in\R\ (j=1,\dots,r)$. 
Under this representation, it is not hard to see that $Q_{\Sigma'}(\gamma,[\sigma_j-\sigma'_j])$ equals to $-a_j$. 
Since $m_i$ and $l_j$ satisfy that $Q_Z(m_i,l_j)$ equals to $\delta_{ij}$, $Q_Z(\gamma,l_j)$ is equal to $a_j$.  
Hence we get
\[
[D(\gamma)_\#(\sigma_j)-\sigma_j']=[\sigma_j-\sigma_j']-Q_Z(\gamma,l_j)\gamma.
\]
By the transposition of the term $\sigma'_j$ on the both sides of the above equality, this induces
\[
[D(\gamma)_\#(\sigma)-\sigma]=-Q_Z(\gamma,l_j)\gamma.
\]
Next, for any $\gamma_1,\dots,\gamma_k\subset\Sigma_{0,r+1}$, we assume
\[
[\left(D(\gamma_k)\circ\dots\circ D(\gamma_1)\right)_\#(\sigma_j)-\sigma_j]=-\underset{s=1}{\overset{k}{\sum}}Q_Z(\gamma_s,l_j)\gamma_s.
\]
We prove the statement by induction on $k$.
 Let $\psi$ be $D(\gamma_k)\circ\dots\circ D(\gamma_1)$ and $\gamma$ a simple closed curve in $\Sigma_{0,r+1}$. 
The restrictions of these diffeomorphisms to $h_1$ are both the identity, so by the Picard-Lefschetz formula, 
\[
[D(\gamma)_\#(\psi_\#(\sigma_j))-\sigma'_j]=[D(\gamma)_\#(\psi_\#(\sigma_j-\sigma_j'))]=[\psi_\#(\sigma_j)-\sigma_j']+
Q_{\Sigma'}(\gamma,[\psi_\#(\sigma_j)-\sigma_j'])
\]
We can write $[\psi_\#(\sigma_j)-\sigma_j']$ as $[\sigma_j-\sigma'_j]+\alpha$ because of the Picard-Lefschetz formula, 
where $\alpha$ is an element of $H_1(\Sigma')$.
 Then we get
\[
[D(\gamma)_\#(\psi_\#(\sigma_j))-\sigma'_j]=[\sigma_j-\sigma'_j]+\alpha-Q_Z(\gamma,l_j)\gamma+Q_{\Sigma'}(\gamma,\alpha)\gamma.
\]
By the transposition, it leads
\[
[D(\gamma)_\#(\psi_\#(\sigma_j))-\sigma_j]=\alpha-Q_Z(\gamma,l_j)\gamma+Q_{\Sigma'}(\gamma,\alpha)\gamma.
\]
On the other hand, $\alpha$ is equal to $[\psi_\#(\sigma_j)-\sigma]$ and then, by the assumption, $\alpha$ is equal to 
$-\underset{s=1}{\overset{k}{\sum}}Q_Z(\gamma_s,l_j)\gamma_s$. Moreover, $Q_{\Sigma_{0,r+1}}$ is 0, then $Q_{\Sigma'}(\gamma,\gamma_s)$ vanishes.
 This implies that $Q_{\Sigma'}(\gamma,\alpha)$ also vanishes.
\end{proof}


\section{Examples}

In this section, we give two examples of caluculation for the signature.
 Let $r$ be a positive integer and consider Lefschetz fibrations $Y_1$ and $Y_2$ with fiber $\Sigma_{0,r+2}$.
 We denote the $i$-th boundary component of $\Sigma_{0,r+2}$ by $\delta_i$ and the boundary of gluing region by Z. 

 The vanishing cycles of the first example $Y_1$ are ordered set of the curves $(\gamma_{1,2},\gamma_{1,3},\dots\\\dots,\gamma_{r-1,r},\gamma_{r-1,r+1},\gamma_{r,r+1})$,
 where $\gamma_{i,j}$ is a curve surrounding $\delta_i$ and $\delta_j$ for $1\leq i<j\leq r+1$ as in Fig \ref{figex}.
 \begin{figure}[ht!]
\labellist
\hair 2pt
\pinlabel $\delta_{r+1}$ [t] at 220 480
\pinlabel $\delta_2$ [t] at 600 480
\pinlabel $\delta_1$ [t] at 410 540
\pinlabel $\delta_{r}$ [t] at 140 332
\pinlabel $\delta_{r-1}$ [t] at 220 185
\pinlabel $\delta_3$ [t] at 680 332
\pinlabel $\delta_4$ [t] at 600 185
\pinlabel $\gamma_{1,r+1}$ [t] at 280 590
\pinlabel $\gamma_{1,2}$ [t] at 505 590
\pinlabel $\delta_0$ [t] at 100 590
\pinlabel $\gamma_{2,3}$ [t] at 730 430
\pinlabel $\gamma_{3,r-1}$ [t] at 410 245
\endlabellist
\centering
\includegraphics[scale=0.3]{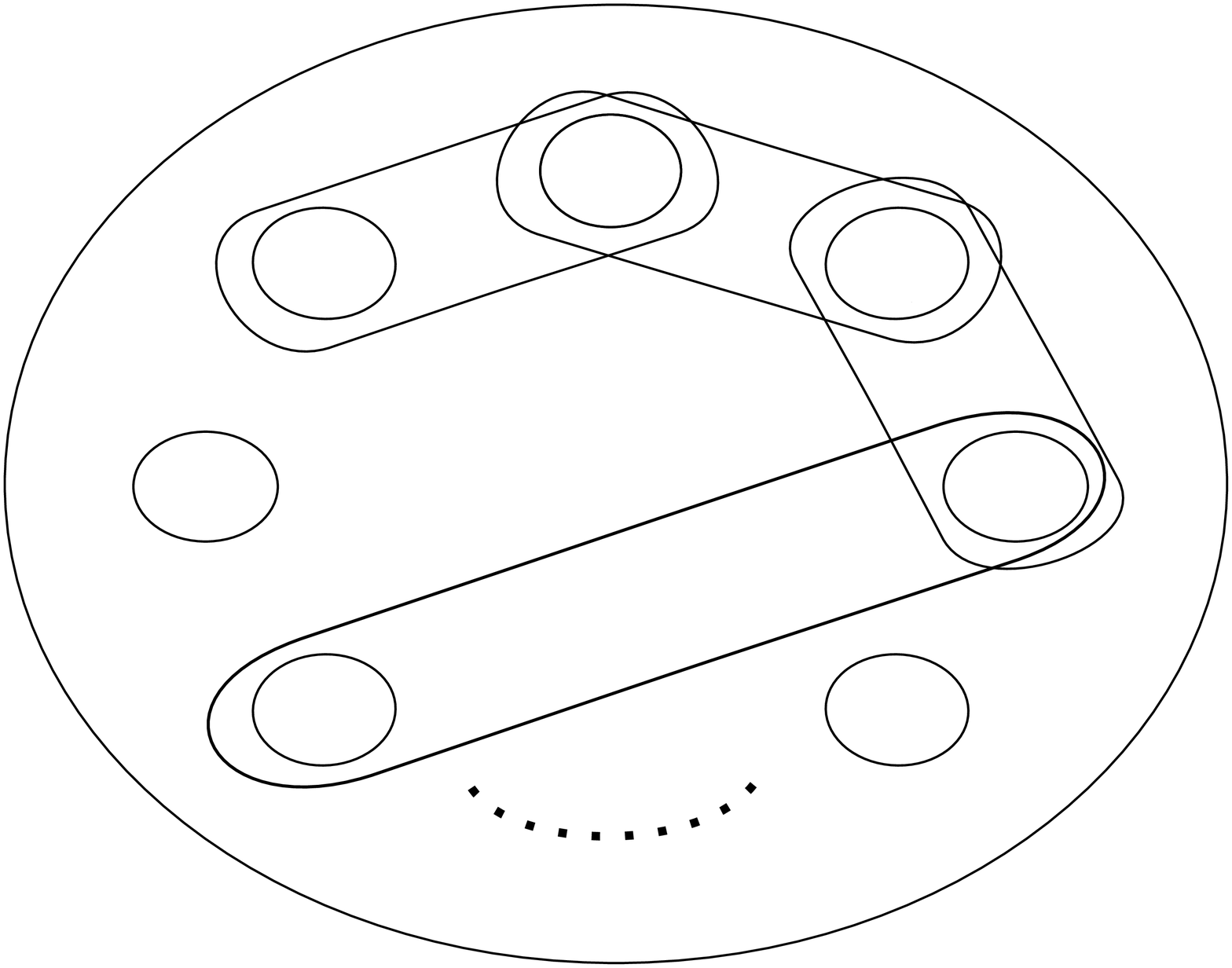}
\caption{}
\label{figex}
\end{figure}
 The homology class of each vanishing cycle $\gamma_{i,j}$ is $\delta_i+\delta_j$ and $Q_Z(\delta_i,l_j)$ is equal to $\delta_{ij}$.
 We get:
\begin{align}
\underset{1\leq i<j\leq r+1}{\sum}Q_Z(\gamma_{i,j},l_k)\gamma_{i,j}
&=\underset{i=1}{\overset{k-1}{\sum}}Q_Z(\gamma_{i,k},l_k)\gamma_{i,k}+
\underset{i=k+1}{\overset{r+1}{\sum}}Q_Z(\gamma_{k,i},l_k)\gamma_{k,1}\nonumber\\
&=\underset{i=1}{\overset{k-1}{\sum}}(\delta_i+\delta_k)+\underset{i=k+1}{\overset{r+1}{\sum}}(\delta_k+\delta_i)\nonumber\\
&=\delta_1+\delta_2+\cdots+r\delta_k+\cdots+\delta_{r+1}.\nonumber
\end{align}
By Lemma \ref{lem1} and Proposition \ref{p2}, the dimension of $\langle \gamma_{1,2},\dots,\gamma_{r,r+1}\rangle$
 is equal to that of
 $\langle \underset{1\leq i<j\leq r+1}{\sum}Q_Z(\gamma_{i,j},l_1)\gamma_{i,j,\dots,\underset{1\leq i<j\leq r+1}{\sum}Q_Z(\gamma_{i,j},l_r+1)\gamma_{i,j}}\rangle$.
 By simply computation, dim$\langle r\delta_1+\cdots+\delta_{r+1},\dots,\delta_1+\cdots+r\delta_{r+1}\rangle$ is equal to $r+1$ and the number of vanishing cycles $r(r+1)/2$.
 Theorem \ref{thm1} implies that
 $\sigma(Y_1)=-r(r+1)/2+r+1=-(r-2)(r+1)/2$. 

Next we choose the vanishing cycles as one $\delta_0$ and $r-1$ $\delta_{i}$ for $i=1,\dots,r+1$.
 The global monodromy of $Y_2$ is represented by $D(\delta_0)\circ D(\delta_1)^{r-1}\circ\dots\circ D(\delta_{r+1})^{r-1}$.
 Because of dim $H_1(\Sigma_{0,r+2})$ is $r+1$ and the number of vanishing cycles is $r^2$,
 the signature of $Y_2$ is equal to $-r^2+r+1$.
\begin{rem}
The global monodromies of these examples are the same.
 This fact is proved by Wajnryb \cite{Waj}.
 Therefore the above two allowable Lefschetz fibrations have same global monodromies but different signatures.
\end{rem}

\begin{center}
{}
\end{center}

\end{document}